\documentclass[a4paper,reqno,11pt]{amsart}

\usepackage[T1]{fontenc}
\usepackage{amssymb,amsmath}
\usepackage{mathtools}
\usepackage[table]{xcolor}
\usepackage[utf8]{inputenc}
\usepackage[unicode=true]{hyperref}
\hypersetup{
  breaklinks=true,
  pdfauthor={Atli Fannar Frankl\'{i}n},
  pdftitle={Pattern Avoiding Permutations as Walks},
  colorlinks=true,
  citecolor=blue,
  urlcolor=blue,
  linkcolor=violet,
  pdfborder={0 0 0}}
\usepackage[parfill]{parskip}
\usepackage{tikz}
\usetikzlibrary{matrix,arrows,patterns,calc,positioning,automata}
\usepackage{hyperref}
\usepackage{comment}
\usepackage{makecell}
\usepackage{multicol}

\newtheorem{theorem}{Theorem}
\newtheorem{corollary}[theorem]{Corollary}
\newtheorem{lemma}[theorem]{Lemma}
\newtheorem{conjecture}[theorem]{Conjecture}
\theoremstyle{definition}
\newtheorem{definition}[theorem]{Definition}

\DeclareMathOperator{\Av}{\operatorname{Av}}

\newcommand\abs[1]{\left|#1\right|}
\newcommand\p[1]{\left(#1\right)}

\graphicspath{{myndir/}}

\title{Pattern avoiding permutations as walks}

\author[A.\ F.\ Frankl\'{i}n]{Atli Fannar Frankl\'{i}n}
\address{Department of Mathematics, University of Iceland, Reykjavik, Iceland}
\email{atlifannar@hi.is}

\begin{document}

\begin{abstract}
  The Stanley-Wilf limit of the pattern $1324$ is known to lie between $10.271$ and $13.5$.
  We obtain lower bounds on this limit by encoding permutations as walks in directed
  graphs: building a permutation by successive insertion of maxima corresponds to
  traversing edges, and the growth rate of walks equals the spectral radius of the
  adjacency matrix. For $1324$, this graph is too large for direct computation,
  so we pass to a quotient graph with weighted edges. Conditional on a natural
  conjecture, this yields a lower bound of $10.418$. 
\end{abstract}

\maketitle
\thispagestyle{empty}

\section{Introduction}

A permutation $\pi$ contains a pattern $\tau$ if some subsequence of $\pi$ has the same relative order as $\tau$, otherwise $\pi$ avoids $\tau$. For example $\pi = 465213$ contains the pattern $\tau = 132$ as the relative order of the elements $465$ is the same as $132$. If we swap the first and third elements such that $\pi= 564213$ then $\pi$ avoids $\tau$. We write $\Av_n(\tau)$ for the set of $\tau$-avoiding permutations of length $n$. The Stanley-Wilf limit of $\tau$ is

\[L(\tau) \coloneqq \lim_{n \rightarrow \infty} \sqrt[n]{\abs{\Av_n(\tau)}}\] 

By the Marcus-Tardos theorem \cite{marcus, arratia}, this limit exists for every pattern $\tau$.

For length $3$ patterns, all Stanley-Wilf limits equal $4$, the growth rate of the Catalan numbers. For length $4$, all patterns have been exactly enumerated and their limits been found \cite{bonaexact, gesselexact} except one: up to symmetry, $1324$ is the only length $4$ pattern whose Stanley-Wilf limit remains unknown. The best current bounds are $10.271 \leq L(1324) \leq 13.5$ \cite{bounds}, with numerical evidence suggesting a value near 11.6 \cite{estimate}.

We develop a method for bounding Stanley-Wilf limits from below by encoding pattern-avoiding permutations as walks in directed graphs. Any permutation can be built by starting from a single element and repeatedly inserting a new maximum. Each insertion corresponds to an edge in a graph whose vertices are equivalence classes of intermediate permutations. The exponential growth rate of walks equals the spectral radius of the adjacency matrix, and the Perron-Frobenius theorem together with the Collatz-Wielandt formula provide tools to bound this from below. To some extent, our approach resembles taht of Albert, Elder, Rechnitzer, Wextcott and Zabrocki \cite{lockseq}: we construct a sequence of increasing lower bounds through computational methods. Our method builds on that paper, but uses the simpler insertion scheme convered by Poh \cite{topinsert} which considers how one can make state machines encoding $1324$-avoiding permutations via repeatedly inserting a new maxmimum element.

For $1324$-avoiding permutations, the natural graph is too large for direct computation, so we pass to a quotient graph with weighted edges. We conjecture that the sum of weighted walks is at most the number of unweighted walks (Conjecture~\ref{conj}); conditional on this, we obtain a lower bound of $10.418$ on $L(1324)$.

Section 2 develops the method through three examples: $\Av(213)$, $\Av(2134)$ and $\Av(3124)$. Section 3 treats $1324$-avoiding permutations, establishing the combinatorial machinery and stating the conjecture. Section 4 discusses open questions.

\section{Motivating examples}

First we summarise notation and definitions for classical permutation patterns, following Bevan  \cite{basicnotation}. Two sequences $a_1, \dots, a_m$ and $b_1, \dots, b_m$ are \textit{order-isomorphic} if $a_i < a_j$ holds if and only if $b_i < b_j$ for all $i, j$. For permutations $\pi = \pi_1 \dots \pi_n$ and $\tau = \tau_1 \dots \tau_m$, $\pi$ \textit{contains the pattern} $\tau$ if there exists a set of indices $i_1, \dots, i_m$ such that $\pi_{i_1}, \dots, \pi_{i_m}$ is order-isomorphic to $\tau_1 \dots \tau_m$. We call such a set of indices an \textit{occurrence} of the pattern $\tau$ in $\pi$. If $\pi$ has no occurrence of $\tau$ then $\pi$ \textit{avoids} $\tau$.

In an effort to bound the growth rate of $1324$-avoiding permutations we will look at encoding the permutations as walks in directed graphs. Traversing an edge in this graph will correspond to inserting a new maximum element in the permutation. Clearly we can construct any given nonempty permutation by inserting a new maximum element over and over, starting from the trivial permutation. The \textit{trivial permutation} is the unique permutation of a single element. To illustrate the construction, we examine $\Av(213)$
as a motivating example. 

As a naïve example we could imagine creating a directed graph with the elements of $\Av(213)$ as vertices, and placing an edge from $\pi$ to $\tau$ if $\tau$ can be obtained by inserting a new maximum element into $\pi$. Our numerical methods will only work for finite graphs, so we will have to define a cutoff $N$, and only work with vertices corresponding to permutations with $\leq N$ elements. Clearly the number of walks in this smaller graph is a lower bound for the walks in the full graph. This reduced graph will still be much too large for our purposes, so we will try to reduce the size of the graph without losing any of the information contained in our walks.

\begin{figure}[h]
\centering
\begin{tikzpicture}
\node[state] (1) at (0, 0) {$1$};

\node[state] (12) at (1, 1) {$12$};
\node[state] (21) at (1, -1) {$21$};

\node[state] (123) at (2.5, 2.5) {$123$};
\node[state] (132) at (2.5, 1) {$132$};
\node[state] (312) at (2.5, -0.5) {$312$};
\node[state] (231) at (2.5, -2) {$231$};
\node[state] (321) at (2.5, -3.5) {$321$};

\draw (1) edge[->] (12);
\draw (1) edge[->] (21);
\draw (21) edge[->] (231);
\draw (21) edge[->] (321);
\draw (12) edge[->] (123);
\draw (12) edge[->] (132);
\draw (12) edge[->] (312);
      
\node (etc) at (3.5, 0) {\Huge $\Rightarrow$};

\node[state] (1b) at (4.5, 0) {$1$};

\node[state] (12b) at (5.5, 1) {$12$};
\node[state] (21b) at (5.5, -1) {$2$};

\node[state] (123b) at (7, 2.5) {$123$};
\node[state] (132b) at (7, 1) {$13$};
\node[state] (312b) at (7, -0.5) {$3$};
\node[state] (231b) at (7, -2) {$23$};
\node[state] (321b) at (7, -3.5) {$3$};

\draw (1b) edge[->] (12b);
\draw (1b) edge[->] (21b);
\draw (21b) edge[->] (231b);
\draw (21b) edge[->] (321b);
\draw (12b) edge[->] (123b);
\draw (12b) edge[->] (132b);
\draw (12b) edge[->] (312b);

\node (etc) at (8, 0) {\Huge $\Rightarrow$};

\node[state] (1c) at (9, 0) {$1$};

\node[state] (12c) at (10, 1) {$12$};

\node[state] (123c) at (11.5, 2.5) {$123$};

\draw (1c) edge[->, bend left=20] (12c);
\draw (1c) edge[loop below] (1c);
\draw (12c) edge[->, bend left=20] (123c);
\draw (12c) edge[->, loop below] (12c);
\draw (12c) edge[->, bend left=20] (1c);
\draw (123c) edge[->, bend left=20] (12c);
\draw (123c) edge[->, loop below] (123c);
\draw (123c) edge[->, bend right=60] (1c);
\end{tikzpicture}
\caption{Construction of $\Av(213)$ graph with cutoff $N = 3$}
\label{g213}
\end{figure}
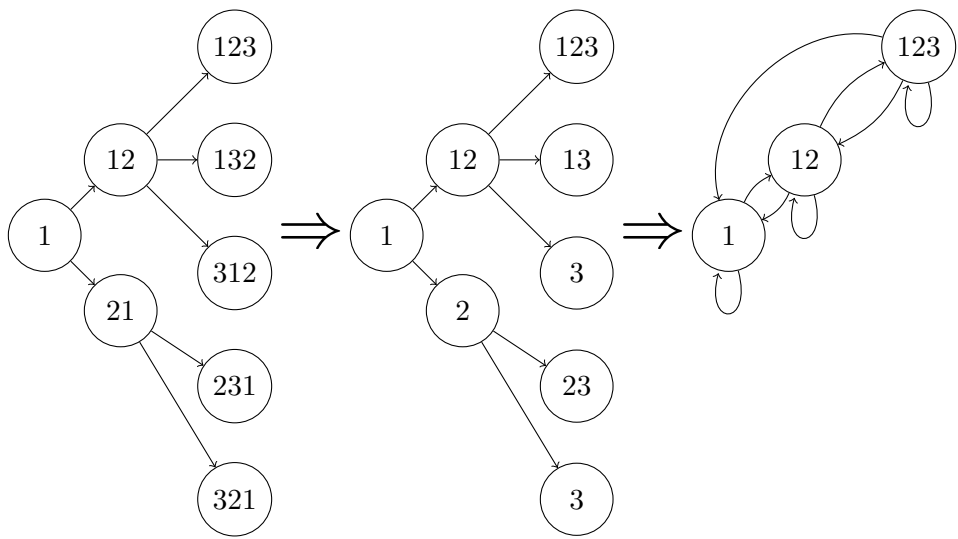

We note that inserting a new maximum element to the right of an instance of the pattern $21$ in a permutation necessarily creates an instance of the pattern $213$. So if we want to avoid the pattern $213$ we have to insert any new maximum elements to the left of all instances of the pattern $21$. This means that for the purposes of our walks, any elements in a permutation after the first instance of $21$ are entirely irrelevant. Thus we can consider only the maximal prefix that is in $\Av(21)$ and then standardise. Standardising a sequence of $n$ distinct values means renaming the elements to obtain the order-isomorphic permutation on $1, 2, \dots, n$. Finally we can equate vertices that correspond to the same permutation, reducing our graph to just having vertices in $\Av(21)$, see Figure~\ref{g213}.

As illustrated by the edge from 123 to 1, we also consider the edges that would have been present were the cutoff value one greater, in case that points back into the lower cutoff. That edge would have gone from $123$ to $4123$, which becomes an edge from $123$ to $1$. Assuming we have no cutoff, the walks in this new graph are in one-to-one correspondence with walks in the original graph, so it suffices to count walks in the new graph. Clearly the number of walks in the cutoff graph is a lower bound for the walks in the full graph. We also see that each vertex corresponds to an increasing permutation $12\dots n$. The outgoing edges from $12 \dots n$ will point to all smaller permutations, the permutation itself and the permutation of one size greater, as inserting a new maximum element at index $i$ will create a permutation of size $i$.

Now that we have our graph and a description of its edges, the next matter is to enumerate walks starting at the vertex $1$. We can consider the adjacency matrix $A$ of our graph. The entries of $A^k$ give us the walks of length $k$ in our graph, so we are interested in the growth rate of the entries of $A^n$ as $n \rightarrow \infty$. To this end we use the Perron-Frobenius theorem \cite{frobenius, perron} . We need the more general form later stated by Frobenius as our later matrices will have many zero entries, and the original version by Perron covers only positive matrices. The portion of the theorem we need can be phrased as follows:

\begin{theorem}[Perron-Frobenius Theorem]
	Let $A$ be an aperiodic irreducible non-negative $N \times N$ matrix with spectral radius $r$ (largest absolute value among all eigenvalues). Then $r\in\mathbb{R}_{>0}$ is an eigenvalue of $A$ which we call the \textit{Perron-Frobenius eigenvalue}. This eigenvalue is simple and both the left and right eigenspaces are one-dimensional. The corresponding left and right eigenvectors are all non-negative, and moreover these are the only eigenvectors of $A$ whose components are all non-negative. If we denote these left and right eigenvectors by $v, w$ and normalise them so $w^Tv = 1$ then $\lim_{n \rightarrow \infty} A^n/r^n = vw^T$.
\end{theorem}

Irreducibility is equivalent to every vertex being able to reach every other vertex, which is called \textit{strong connectivity}, while for the aperiodicity it will suffice that some vertex has an edge pointing to itself, called a \textit{loop}. Since all vertices can reach $1$ in a single step and $1$ has a loop, this is satisfied. The theorem also tells us that the number of walks starting at $1$ of length $n$ must grow like $\lambda^n$ where $\lambda$ is the Perron-Frobenius eigenvalue. Thus each cutoff $N$ gives us some Perron-Frobenius eigenvalue $\lambda_N$ that will be a lower bound for the Stanley-Wilf limit.

To determine the Perron-Frobenius eigenvalue of a graph we make use of the Collatz-Wielandt formula \cite{collatz, wielandt}. The result can be stated as follows:

\begin{theorem}[Collatz-Wielandt Formula]
	Let $A$ be a matrix with Perron-Frobenius eigenvalue $r$. For a non-zero vector $v$ with non-negative entries, the minimum value of $(Av)_i / v_i$ taken over $i$ such that $v_i \neq 0$ is $\leq r$.
\end{theorem}

Thus if we produce a vector $v$ such that $Av$ is componentwise $\geq \rho v$ for some real value $\rho$, then the Perron-Frobenius eigenvalue is at least $\rho$. In this case we will show that the limit is at least $\rho = 4$, which is of course also its true value.

If we take $v = (1, 1, 2/3, (2/3)^2, (2/3)^3, \dots)$, cutting the vector off to the right length, we claim we get our desired bound. The sum of the coordinates of $v$ is then $\alpha \coloneqq 4 - 2^{N-1}/3^{N-2}$ for cutoff $N$. By direct calculation we get that 
\[Av = \p{\alpha, \alpha, \alpha - 1, \alpha -  1 - \frac{2}{3}, \alpha -  1 - \frac{2}{3} - \p{\frac{2}{3}}^2, \dots}\]
As $N \rightarrow \infty$, the ratios $(Av)_i / v_i$ all approach $4$ from below. Therefore it will be a lower bound for the growth rate, so the growth rate must be at least $4$. Thus in the case of $\Av(213)$ we recover the exact Stanley-Wilf limit.

Next we consider a bigger example, $\Av(2134)$. This is known to have a Stanley-Wilf limit of $9$ due to Regev \cite{regev}. If we apply the same kind of construction as above we will have a graph with vertices in $\Av(213)$ up to some cutoff. There are however some small changes that have to be made, see Figure~\ref{g2134v1}.

\begin{figure}[h]
	\centering
	\begin{tikzpicture}
		\node[state, red] (1) at (0, 0) {$1$};
		
		\node[state, red] (12) at (-2, -1) {$12$};
		\node[state] (21) at (2, -1) {$21$};
		
		\node[state, red] (123) at (-3, -2.5) {$123$};
		\node[state] (132) at (-1.5, -2.5) {$132$};
		\node[state] (312) at (0, -2.5) {$312$};
		\node[state] (231) at (1.5, -2.5) {$231$};
		\node[state] (321) at (3, -2.5) {$321$};
		
		\draw (1) edge[->] (12);
		\draw (1) edge[->] (21);
		\draw (21) edge[->] (231);
		\draw (21) edge[->, bend left = 20] (321);
		\draw (21) edge[->, loop right] (21);
		\draw (12) edge[->] (123);
		\draw (12) edge[->] (132);
		\draw (12) edge[->] (312);
		
		\draw (132) edge[->, loop below] (132);
		\draw (312) edge[->, loop below] (312);
		\draw (231) edge[->, loop below] (231);
		\draw (321) edge[->, loop below] (321);
		\draw (312) edge[->] (21);
		\draw (321) edge[->, bend left = 20] (21);
	\end{tikzpicture}
	\caption{Graph of $\Av(2134)$ with cutoff $N = 3$, version one}
	\label{g2134v1}
\end{figure}
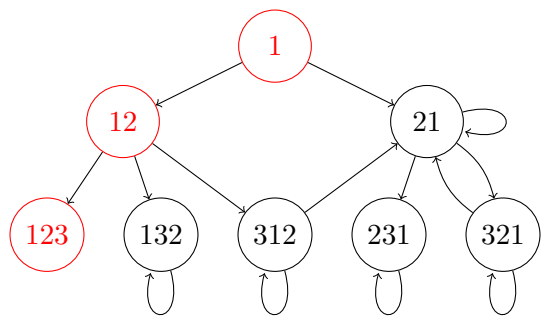

The vertices corresponding to increasing permutations are not strongly connected to the rest of the graph, causing an issue. Suppose we have two permutations $\pi, \tau \in \Av(213)$, each with at least one inversion, and we want to find a path from $\tau$ to $\pi$ in the graph. Then we can insert each element of $\pi$ at the start of $\tau$, and then a new maximum after that. For example if we started at $\tau = 1243$ and our target were $\pi= 132$ we would make the walk 
\[1243 \rightarrow 51243 \rightarrow 561243 \rightarrow 5761243 \rightarrow 57681243\]
$768$ is an occurrence of $213$, so the last vertex is actually $576$, which is standardised as $132$. Thus aside from the increasing permutations, any vertex can reach any other vertex. As we are calculating a lower bound, we can fix this problem by simply omitting the increasing permutations.

We make one more modification to our construction process that did not come up for $\Av(213)$. Suppose we have some walk starting at a vertex corresponding to the permutation $\pi$. If $\pi$ is order-isomorphic to some prefix of another permutation $\tau$, then all the insertions into $\pi$ could be done on $\tau$ as well. This way we can map the walks starting at $\pi$ injectively to the walks starting at $\tau$. Therefore if $\pi$ has an edge to $\tau$ in the full graph but $\tau$ is not contained in the cutoff graph, we could retain more information by having an edge from $\pi$ to the maximum prefix of $\tau$ that is still in the graph. For example if our graph has cutoff $N = 3$ and we are looking at $\pi= 321$ and insert a maximum to get $3421$, then this is not in our graph. So we instead put an edge from $\pi$ to $342$, which is standardised as $231$. This modified graph is shown in Figure~\ref{g2134v2}. Some edges that do not contribute to strong connectivity are omitted for clarity. 

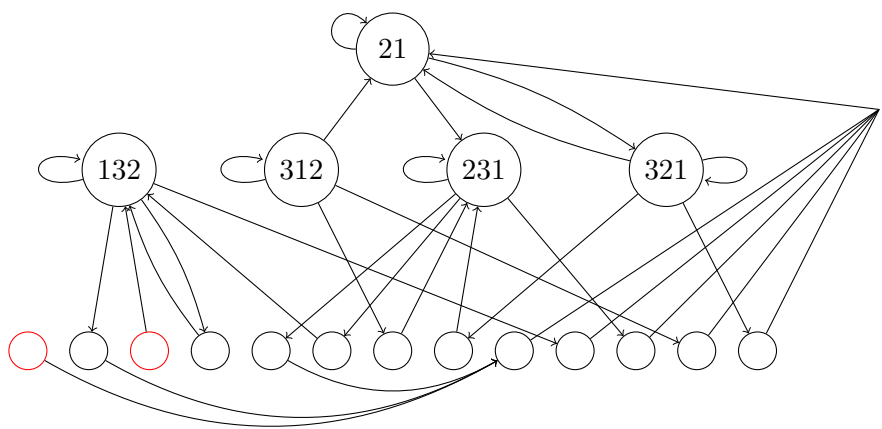
\begin{figure}[h]
	\centering
	\begin{tikzpicture}[scale=0.8]
		\node[state] (21) at (0, 0) {$21$};
		
		\node[state] (132) at (-4.5, -2) {$132$};
		\node[state] (312) at (-1.5, -2) {$312$};
		\node[state] (231) at (1.5, -2) {$231$};
		\node[state] (321) at (4.5, -2) {$321$};
		
		\node[state,minimum size=0.5cm, red] (1243) at (-6, -5) {};
		\node[state,minimum size=0.5cm] (1342) at (-5, -5) {};
		\node[state,minimum size=0.5cm, red] (1423) at (-4, -5) {};
		\node[state,minimum size=0.5cm] (1432) at (-3, -5) {};
		\node[state,minimum size=0.5cm] (2341) at (-2, -5) {};
		\node[state,minimum size=0.5cm] (2431) at (-1, -5) {};
		\node[state,minimum size=0.5cm] (3412) at (0, -5) {};
		\node[state,minimum size=0.5cm] (3421) at (1, -5) {};
		\node[state,minimum size=0.5cm] (4123) at (2, -5) {};
		\node[state,minimum size=0.5cm] (4132) at (3, -5) {};
		\node[state,minimum size=0.5cm] (4231) at (4, -5) {};
		\node[state,minimum size=0.5cm] (4312) at (5, -5) {};
		\node[state,minimum size=0.5cm] (4321) at (6, -5) {};
		
		\draw (21) edge[->, in=135, out=180, looseness=4] (21);
		\draw (21) edge[->] (231);
		\draw (21) edge[->, bend left=10] (321);
		
		\draw(132) edge[->, loop left] (132);
		\draw(132) edge[->] (1342);
		\draw(132) edge[->, bend left=10] (1432);
		\draw(132) edge[->] (4132);
		
		\draw(312) edge[->] (21);
		\draw(312) edge[->, loop left] (312);
		\draw(312) edge[->] (3412);
		\draw(312) edge[->] (4312);
		
		\draw(231) edge[->, loop left] (231);
		\draw(231) edge[->] (4231);
		\draw(231) edge[->] (2431);
		\draw(231) edge[->] (2341);
		
		\draw (321) edge[->, bend left=10] (21);
		\draw (321) edge[->, loop right] (321);
		\draw (321) edge[->] (4321);
		\draw (321) edge[->] (3421);
		
		\draw (8,-1) edge[->] (21);
		\draw (4321) -- (8,-1);
		\draw (4312) -- (8,-1);
		\draw (4231) -- (8,-1);
		\draw (4132) -- (8,-1);
		\draw (4123) -- (8,-1);
		\draw (3421) edge[->] (231);
		\draw (3412) edge[->] (231);
		\draw (2431) edge[->] (132);
		\draw (2341) edge[->, bend right=30] (4123);
		\draw (1432) edge[->, bend left=10] (132);
		\draw (1423) edge[->] (132);
		\draw (1342) edge[->, bend right=30] (4123);
		\draw (1243) edge[->, bend right=30] (4123);
 	\end{tikzpicture}
	\caption{Filtered $\Av(2134)$, cutoff $N = 4$, version two}
	\label{g2134v2}
\end{figure}

Two vertices are not strongly connected to the rest, the vertices corresponding to $1423$ and $1243$. But if we were to go up to cutoff $N = 5$ then they would become connected, so at any given point we can just omit the vertices that are not yet strongly connected. This is similar to how $132$ would not be strongly connected to the other length $3$ permutations without the length $4$ permutations.

This construction yields an exponential number of vertices in the cutoff $N$, so we would like to reduce their number somewhat. To this end we consider what happens when a new maximum element is inserted into a $213$-avoiding permutation.

If a new maximum element is inserted into $\pi \in \Av(213)$ one of two things happens. That insertion can happen in an initial increasing run, changing the length of that run and increasing the length of the permutation by 1. Otherwise the insertion happens after a descent so we will have to delete everything after the inserted element, and then the inserted element as well. This decreases the size of the permutation but leaves the initial increasing run unchanged. So if we know the length of the initial increasing run and the length of the permutation before and after the insertion, we can recover where the element was inserted. Applying this inductively, if we know the sequence of permutation sizes and initial run lengths, we can recover the walk in the graph. We define the partition of $\Av_n(213)$ into sets with initial run of length $r$, which we denote $B(n, r)$. Then it suffices to consider the quotient graph with respect to this partition, where we collapse all vertices in the same set $B(n, r)$ into the same vertex. This way, instead of working with a graph with about $4^N$ vertices for a cutoff $N$, we can have about $N^2$ vertices. Now numerical computations suddenly become much more feasible. In terms of the vector we try to find for the Collatz-Wielandt formula, this is equivalent to deciding that we will give all coordinates in the same part $B(n, r)$ get the same value. An example graph can be seen in Figure~\ref{bnrgraph}. The node for $B(5, 4)$ gets omitted because it would not be strongly connected to the rest of the graph.
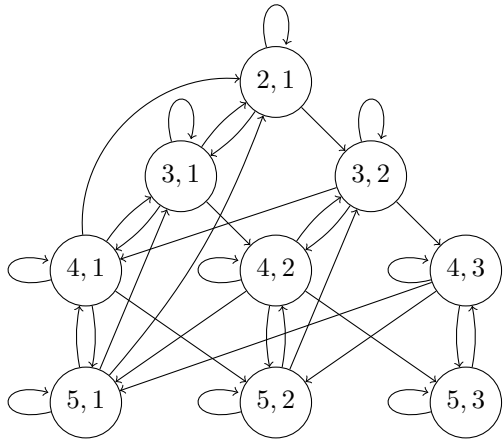
\begin{figure}[h]
	\centering
	\begin{tikzpicture}
		\node[state, minimum size=0.75] (21) at (0, 0) {\small $2,1$};
		\node[state, minimum size=0.75] (31) at (-1.25, -1.25) {\small $3,1$};
		\node[state, minimum size=0.75] (32) at (1.25, -1.25) {\small $3,2$};
		\node[state, minimum size=0.75] (41) at (-2.5, -2.5) {\small $4,1$};
		\node[state, minimum size=0.75] (42) at (0, -2.5) {\small $4,2$};
		\node[state, minimum size=0.75] (43) at (2.5, -2.5) {\small $4,3$};
		\node[state, minimum size=0.75] (51) at (-2.5, -4.25) {\small $5,1$};
		\node[state, minimum size=0.75] (52) at (0, -4.25) {\small $5,2$};
		\node[state, minimum size=0.75] (53) at (2.5, -4.25) {\small $5,3$};
		
		\draw (21) edge[->, bend left=10] (31);
		\draw (21) edge[->] (32);
		\draw (21) edge[->, loop above] (21);
		
		\draw (31) edge[->, bend left=10] (41);
		\draw (31) edge[->] (42);
		\draw (31) edge[->, bend left=10] (21);
		\draw (31) edge[->, loop above] (31);
		
		\draw (32) edge[->] (41);
		\draw (32) edge[->, bend left=10] (42);
		\draw (32) edge[->] (43);
		\draw (32) edge[->, loop above] (32);
		
		\draw (41) edge[->, bend left=10] (51);
		\draw (41) edge[->] (52);
		\draw (41) edge[->, bend left=50] (21);
		\draw (41) edge[->, bend left=10] (31);
		\draw (41) edge[->, loop left] (41);
		
		\draw (42) edge[->] (51);
		\draw (42) edge[->, bend right=10] (52);
		\draw (42) edge[->] (53);
		\draw (42) edge[->, bend left=10] (32);
		\draw (42) edge[->, loop left] (42);
		
		\draw (43) edge[->] (51);
		\draw (43) edge[->] (52);
		\draw (43) edge[->, bend right=10] (53);
		\draw (43) edge[->, loop left] (43);
		
		\draw (51) edge[->, bend right=10] (21);
		\draw (51) edge[->] (31);
		\draw (51) edge[->, bend left=10] (41);
		\draw (51) edge[->, loop left] (51);
		
		\draw (52) edge[->] (32);
		\draw (52) edge[->, bend right=10] (42);
		\draw (52) edge[->, loop left] (52);
		
		\draw (53) edge[->, bend right=10] (43);
		\draw (53) edge[->, loop left] (53);
	\end{tikzpicture}
	\caption{Quotient graph on $B(n, r)$ for cutoff $N = 5$}
	\label{bnrgraph}
\end{figure}

To do the numerical computations we need three things. We need to know the sizes of the sets $B(n, r)$, the number of edges from one vertex $B(n, r)$ to another vertex $B(m, s)$ and we need some numerical method to determine the Perron-Frobenius eigenvalue. This eigenvalue can then be validated using the Collatz-Wielandt formula.

Using the standard bijection between Dyck paths and $213$-avoiding permutations, we have due to Kreweras \cite{kreweras} that $|B(n, r)| = T_{n-1,n-r}$ where 
\[T_{n,k} = \frac{n-k+1}{n+1} \binom{n+k}{n}\]
which will make an appearance again later. From earlier we also know that for any $\pi \in B(n, r)$ there will be $n + 1$ edges out, each of which points to one of the sets $B(n + 1, 1), B(n + 1, 2), \dots, B(n + 1, r + 1)$ or $B(r + 1, r), B(r + 2, r), \dots, B(n, r)$. Thus if $B(m, s)$ is among these $n + 1$ sets, there are $|B(n, r)|$ edges to it from $B(n, r)$ and otherwise zero edges.

Lastly we will use the power method \cite{mises} to calculate the Perron-Frobenius eigenvalue. We start with the vector $v$ of all ones, then repeatedly replace $v$ with $Av$, then normalise. This can be repeated until we get acceptably close to the true eigenvalue, and in the case of a Perron-Frobenius eigenvalue this will always converge to the right value. Now we can use a computer to calculate the growth rate for larger and larger cutoffs $N$, getting Figure~\ref{213lam}.

\begin{figure}[h]
	\centering
	\includegraphics[width=\textwidth]{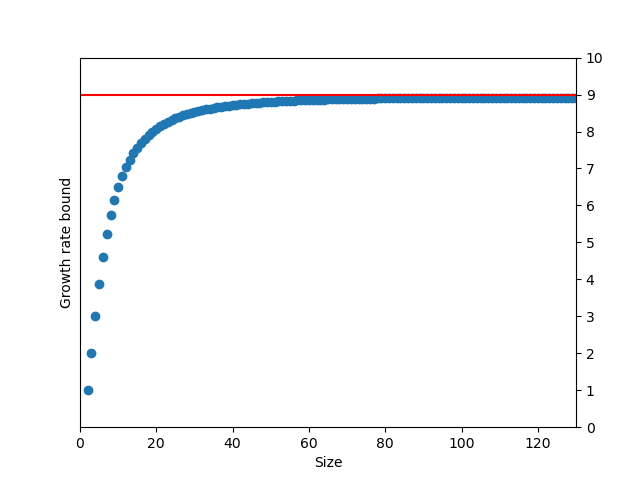}
	\caption{Stanley-Wilf estimate of $2134$ as a function of cutoff}
	\label{213lam}
\end{figure}

This gives a lower bound of $8.90952$ on the growth rate, within $1\%$ of the true value. The computation was simply done on a personal computer in about a minute. A more serious computation could likely get a fair bit closer to the true value. We also note that this is seemingly not the only partition that will work, the same example can be worked through when partitioning on the value of the last element of the permutation rather than the length of the initial run.

As a last motivating example we can do something similar for $\Av(3124)$. Then we are looking at a graph with vertices in $\Av(312)$ and we want to find some analogous partition that allows us to recover the walks. Many different statistics can be considered, but one that works to reduce the graph somewhat is to partition by the descent set. The descent set of a permutation $\pi$ is the set of indices $i$ such that $\pi_i > \pi_{i+1}$. Thus we denote the subset of $\Av_n(312)$ with descent set $S$ by $C(n, S)$. For permutations of size $n$ there are $2^n$ possible descent sets. Thus the size of the graph would grow like $2^N$ rather than $4^N$ as a function of the cutoff $N$. While this is still exponential, it is an improvement. 

Again we have to determine the sizes $|C(n, S)|$ and the number of edges going from some $C(n, S)$ to some $C(m, T)$. Explicit formulas were not derived, but recurrences were sufficient for computing the values on a computer. For brevity we will omit some details here, but $|C(n, S)|$ can be calculated using a three-dimensional push-forward recurrence for a given size $n$ and descent set $S$. Push-forward means that instead of calculating $C(n, S)$ based on smaller cases, each case $C(n, S)$ contributes something to later cases as we calculate them in some order. Then when we arrive at the case $C(n, S)$ it has the correct value stored. Similarly we omit details on the number of edges between sets, but it follows the same pattern as earlier cases. Finally, we must show the walks in the quotient graph are one-to-one with walks in the original graph. That is, given a sequence of sizes and descent sets, we can recover the original sequence of permutations in $\Av(213)$.

We can show this by induction. At the start we know what permutation we have, as it is the trivial permutation, so we just have to show that if we know $\pi$ and its descent set and size after inserting a new maximum, we know what permutation we have after the insertion. If we insert the maximum before some ascent in $\pi$, then we will remove every element after the ascent and the ascent top after the insertion. So we know in what decreasing run of the permutation the insertion was made. But then we can tell where the insertion was made due to the new ascent, which we can read from the descent set. Therefore we can recover where the insertion was made, so the walks are one-to-one between the full graph and the quotient graph.

\begin{figure}[ht]
	\centering
	\includegraphics[width=\textwidth]{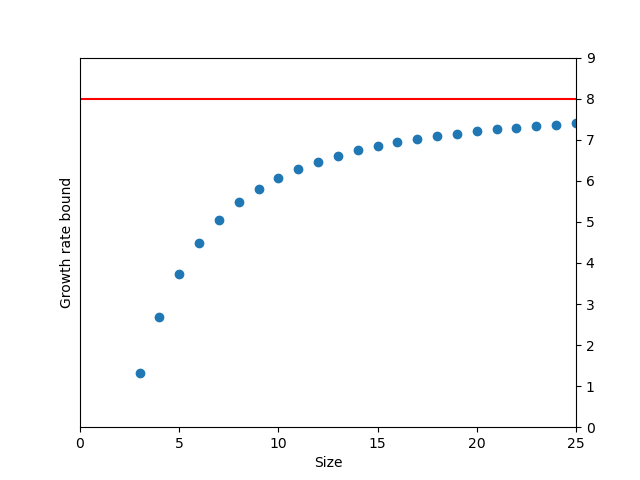}
	\caption{Stanley-Wilf estimate of $3124$ as a function of cutoff}
	\label{312lam}
\end{figure}

All in all this gets us Figure~\ref{312lam}. The true Stanley-Wilf limit is known to be $8$, due to Bóna \cite{bonaexact}. We achieve the bound $7.40341$, which again could be improved with more computation time and more memory.

\section{1324-avoiding Graph}

We now formalise the same construction as in the motivating examples.

\begin{definition}
	Let $\pi$ be a permutation
	whose last element is fixed. The avoider graph of $\pi$ with cutoff $N$ has as vertices
	all nonempty permutations of size at most $N$ avoiding $\pi$. For each permutation $\tau$
	in the graph, there is a directed edge to every permutation obtained by inserting
	a new maximum element into $\tau$ , removing the minimal suffix to avoid $\pi$ and
	maintain size $\leq N$ , then standardising. We denote the number of walks starting from the trivial permutation with $k - 1$ steps with $W_{n, k}(\pi)$.
\end{definition}

The shift in the definition of $W_{n,k}(\pi)$, making it correspond to $k - 1$ steps, is so that $W_{n,n}(\pi) = \abs{\Av_n(\pi)}$. The number of vertices of our graph will grow like $4^N$ when we look at the Stanley-Wilf limit of $1324$, so our numerical methods can only go up to about a cutoff of $N = 17$ before this becomes infeasible due to computational limitations. This small cutoff only yields a lower bound of $8.18$, which is not good enough. We want to do something similar to the grouping $B(n, r)$ and $C(n, S)$ but it seems we cannot get such a lossless partition in this case. For brevity we introduce the notation $\pi \rightarrow \tau$ for $\pi, \tau \in \Av(132)$, denoting that there is an edge from $\pi$ to $\tau$ in the avoider graph with no cutoff. This means that we start with $\pi$ and insert a new maximum to obtain $\pi'$. Next we remove the minimal suffix such that $\pi'$ becomes $132$-avoiding. Finally we standardise and if this results in $\tau$, then $\pi \rightarrow \tau$. 

\begin{theorem}
The multiset $\{|\tau| : \pi \rightarrow \tau\}$ uniquely determines $\pi$ for any $\pi \in \Av(132)$.
\end{theorem}

\begin{proof}
	Let $S$ be the multiset $\{|\tau| : \pi \rightarrow \tau\}$ for some $\pi$. The multiset $S$ already determines the size of the permutation $\pi$. Therefore we can instead work with $S' \coloneqq \{|\pi| - x : x \in S\} \setminus \{0\}$, the multiset of how many elements get removed. The value $|\pi|$ always appears in $S$ at least twice, for the insertion at the front and at the back, so we can always remove one zero. Removing this one zero means all the elements of $S'$ correspond to an insertion after an element in $\pi$.
	
	We split $132$-avoiding permutations on the maximum value using the classical decomposition. This lets us interpret the permutations as binary trees, letting the trivial permutation correspond to a single node and then using the decomposition. Then everything left of the maximum becomes the left subtree of the root of the binary tree, and everything to the right the other subtree.
	
	Suppose $\pi$ has a maximum at index $m$. Let $\pi_L$ be the permutation on the elements before index $m$ in $\pi$ and $\pi_R$ be the permutation on the elements after. If we insert a new maximum before $m$ and $m \neq 1$ then the new permutation will have an occurrence of $132$. To remove this occurrence by removing a suffix, everything after and including the old maximum must be removed. Therefore the values the elements left of index $m$ contribute to $S'$ are simply $\{|\tau|  + |\pi_R| : \pi_L \rightarrow \tau \}$.
	
	Consider next an insertion after index $m$. Since $\pi$ avoids $132$, every element of $\pi_L$ is larger than every element of $\pi_R$. So even after inserting a new maximum we never have an occurrence of $132$ that involves elements before the maximum. Therefore the values the elements right of the index $m$ contribute to $S'$ are $\{|\tau| : \pi_R \rightarrow \tau \}$.
	
	With this information we can present the map taking $\pi$ to its multiset $S'$ in terms of the binary tree corresponding to $\pi$. We create a map $g$ that assigns a non-negative integer to each node of a binary tree, and let $S'$ be the image of the nodes of the whole binary tree under $g$. For a node $v$ we consider the path $P$ from $v$ to the root, including both $v$ and the root. Then we define $g(v)$ as the sum of size of the right subtree of $w$, plus one for $w$ itself, over all $w$ such that both $w$ and its left child are on the path $P$. The multiset of these numbers will then give $S'$, see Figure~\ref{binsetmap}. A more imperative way to view $g$ is to think of starting at the root and walking down to the vertex $v$. Any time you would move from a parent $p$ to its left child, you add the size of the right subtree of $p$ plus one to a running tally. The total sum you get at the end is then $g(v)$.
	
	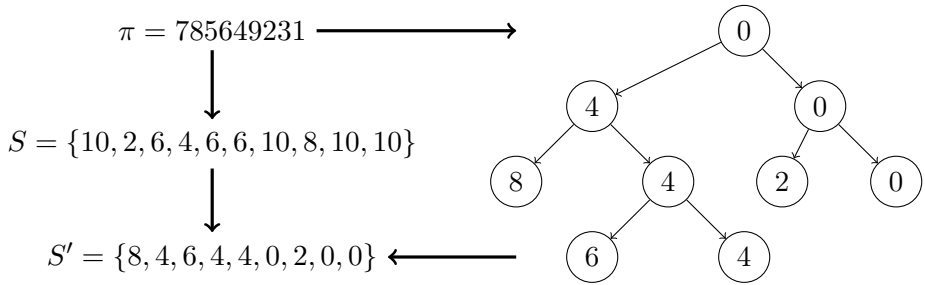
\begin{figure}[h]
		\centering
		\begin{tikzpicture}
			\node[state, minimum size=0.75] (1) at (0, 0) {0};
			\node[state, minimum size=0.75] (2) at (-2, -1) {4};
			\node[state, minimum size=0.75] (3) at (1, -1) {0};
			\node[state, minimum size=0.75] (4) at (-3, -2) {8};
			\node[state, minimum size=0.75] (5) at (-1, -2) {4};
			\node[state, minimum size=0.75] (6) at (2, -2) {0};
			\node[state, minimum size=0.75] (7) at (-2, -3) {6};
			\node[state, minimum size=0.75] (8) at (0, -3) {4};
			\node[state, minimum size=0.75] (9) at (0.5, -2) {2};
			
			\draw (1) edge[->] (2);
			\draw (1) edge[->] (3);
			\draw (2) edge[->] (4);
			\draw (2) edge[->] (5);
			\draw (3) edge[->] (6);
			\draw (5) edge[->] (7);
			\draw (5) edge[->] (8);
			\draw (3) edge[->] (9);
			
			\node (pi) at (-7, 0) {$\pi = 785649231$};
			\node (S1) at (-7, -1.5) {$S = \{10, 2, 6, 4, 6, 6, 10, 8, 10, 10\}$};
			\node (S2) at (-7, -3) {$S' = \{8, 4, 6, 4, 4, 0, 2, 0, 0\}$};
			
			\draw (pi) edge[->, very thick] (S1);
			\draw (pi) edge[->, very thick] (-3, 0);
			\draw (S2) edge[<-, very thick] (S1);
			\draw (S2) edge[<-, very thick] (-3, -3);
		\end{tikzpicture}
		\caption{A permutation, its binary tree and multisets $S, S'$}
		\label{binsetmap}
	\end{figure}
	
	This formulation will allow us to present an inverse, constructing the permutation $\pi$ from the multiset $S'$. But first we need to prove three properties of $g$.
	
	The first one is that $g$ is injective on the leaves. Suppose this is not true and two leaves $v, w$ are mapped to the same value. Consider their lowest common ancestor $\ell$, if this is not the root, we can consider instead the subtree rooted at $\ell$. We can do this because any node outside this subtree either contributes to both $g(v)$ and $g(w)$ or neither. Thus we may assume without loss of generality that $v$ is in the left subtree of the root and $w$ in the right. But $g(w)$ is at most the size of the right subtree, not including the root, and $g(v)$ is at least the size of the right subtree plus the root. This is clearly a contradiction, so all leaves are assigned different values.
	
	The second property we need is that the reverse preorder will visit the nodes in increasing order of the value they map to. The reverse preorder is a recursive procedure for enumerating the vertices of a binary tree. Given a node it enumerates that node, then recursively enumerates the right subtree, then finally recursively enumerates the left subtree. So in Figure~\ref{binsetmap} it would enumerate the values $0, 0, 0, 2, 4, 4, 4, 4, 6, 8$ starting from the root. To prove this we consider some node $v$ and its left and right subtrees. The path from the root to some node in either subtree will contain more nodes, so $g(v)$ cannot be greater than $g(w)$ for $w$ in either subtree. Any $g(w)$ for $w$ in the right subtree will only exceed $g(v)$ by at most the size of the right subtree, while $g(u)$ for $u$ in the left one are at least $g(v)$ plus the size right of the subtree and one more. Thus all the values in the right subtree come before the values in the left subtree, giving us our desired result.
	
	The third and last property we need is that if $v$ is among the nodes that contributes to the sum $g(w)$ for some node $w$, then $g(v) < g(w)$. But if $g(v) < g(w)$, there is some node $u$ on the path from $w$ to the root such that $w$ is in the left subtree of $u$ and $v$ in the right. And by the second property, this means $g(v) < g(w)$.
	
	We consider the values of $S$ in increasing order, denoting them $s_1 \leq s_2 \leq \dots \leq s_k$, and construct the tree one node at a time. Denote the node corresponding to $s_i$ by $v_i$, so we want to satisfy $g(v_i) = s_i$. Because of the third property we proved, we can do this one node at a time, since all the nodes $v_j$ contributing to $s_i$ must have $j < i$. To start we must have $s_1 = 0$, which will correspond to the root $v_1$. As we consider adding $v_i$ to the tree, we know it comes next after $v_{i-1}$ in the reverse preorder traversal of the tree. Thus it is either the right child of $v_i$, or the left child of some node on the path from $v_{i-1}$ to the root.
	
	Consider just the tree consisting of the path from $v_{i-1}$ to the root, then add a left and right child to all nodes in that path where one is not already present. Then all the leaves of this tree are our possible insertion points for $v_i$. But we already know that all leaves correspond to different values, so only one of them can be $s_i$. This uniquely determines where we place $v_i$, before then moving onto $v_{i+1}$ and so on. Finally the third property we proved ensures that $g(v_i) = s_i$ continues to be true as we insert $v_{i+1}, v_{i+2}, \dots$.
	
	Thus we have a binary tree which maps to our set $S$ and it is unique as each node insertion is forced. The $132$-avoiding permutation is uniquely determined as well then, completing our proof.
\end{proof}

This theorem shows we cannot get a partition like in the earlier two cases, where we could massively shrink the number of vertices while maintaining all walks. We will still partition the permutations, but we have to accept getting something approximate rather than exact. But this can still possibly provide a lower bound.

But first one can ask the question of how to find a good partition to try. Ultimately we want to compute some set of weights $w_{\pi}$ on the permutations such that 
\[ \sum_{\tau \rightarrow \pi} \frac{1}{\rho} w_{\tau} \geq w_{\pi}\]
Then the Collatz-Wielandt formula would ensure that $\rho$ is a lower bound for 
the Stanley-Wilf limit of $1324$. Furthermore in the limit of the power method we'd have equality instead of inequality, as we converge to the Perron-Frobenius eigenvector. While these $w_{\pi}$ cannot be calculated directly, as far as we know, a very similar set of weights $w'_{\pi}$ can be found computationally. The set of weights satisfying
\[ \sum_{\tau \rightarrow \pi} \frac{1}{|\tau|} w'_{\tau} = w'_{\pi}\]
is exactly the stationary distribution of the random walk on our graph.

A random walk is a random process, a discrete time Markov chain, in which each state is a vertex of our graph. At any given point the process chooses a neighbour of that vertex uniformly at random and transitions to that state, repeating this infinitely. The stationary distribution then tells us the proportion of time the process spends in any given vertex as time goes to infinity. Because our graph is strongly connected and has a loop a stationary distribution will exist, and can be computed in terms of recurrence times \cite{grimm}. The recurrence time of a vertex $v$ is the expected number of steps it will take to return to $v$ during our random process if we start the process from $v$.

One can calculate that $w'_{\pi}$ appears to only depend on the number of elements of $\pi$ and the number of right-to-left maxima $\pi$. In fact after rescaling the weights $w'_{\pi}$ it appears that they are given by $f(\pi)/|\pi|!$ where $f(\pi)$ is the number of non-right-to-left-maxima. As we will be working with such elements a fair bit going forward, we will call non-right-to-left-maxima \textit{short} values. For $2134$-avoiders, this approach seemingly yields the formula $(n-\ell)/n!$ after scaling, where $\ell$ denotes the last element. The $3124$-avoiding permutations appear to have no such simple formula, though perhaps one can be found in terms of the descent set.

\begin{theorem}
The set of $132$-avoiding permutations of length $n$ with $k$ short values, which we denote $A(n, k)$, has $T_{n-1, k}$ elements where 

\[T_{n, k} = \frac{n - k + 1}{n + 1} \binom{n + k}{n}\]

These values are in the OEIS \cite{oeis} under \texttt{A009766}.
\end{theorem}

\begin{proof}
Suppose we have $\pi \in A(n, k)$. We use the classical decomposition of $132$-avoiding permutations, splitting on the maximum element. Let $\tau_1$ denote everything to the left of the maximum of $\pi$, and $\tau_2$ everything to the right. Every right-to-left maxima of $\tau_2$ will continue to be a right-to-left maxima in $\pi$, but the ones in $\tau_1$ will not because they are to the left of the maximum element of $\pi$. Thus,
\[|A(n, k)| = \sum_{m = 1}^{n-1} C_m |A(n - m - 1, k - m + 1)|\]
where $C_m$ is the $m$-th Catalan number. Using the formula $C_m = \frac{1}{m+1}\binom{2m}{m}$ the desired result follows easily by induction.
\end{proof}

As we saw in the motivating examples we still need one more thing before we can compute the bound with respect to the partition into the sets $A(n, r)$. We need to know the number of edges going from $A(n, r)$ to $A(m, s)$ in the graph for any given integers $n, r, m, s$. But to get there we first need a lemma.

\begin{lemma}\label{differentanr}
	For a permutation $\pi \in \operatorname{Av}_n(132)$, the $n + 1$ permutations $\tau$ such that $\pi \rightarrow \tau$ all belong to different sets $A(m, s)$.
\end{lemma}

\begin{proof}
	Suppose inserting a new maximum at index $i$ and at index $j$ results in the same number of elements in the output permutations $\tau_1$ and $\tau_2$. Then it suffices to show that $\tau_1$ and $\tau_2$ have a different number of left-to-right maxima. Without loss of generality we assume $i < j$. This means the elements $\pi_1, \dots, \pi_{i-1}$ are all greater than the elements $\pi_i, \pi_{i+1}, \dots, \pi_{j - 1}$, as otherwise an occurrence of $132$ would form upon inserting a new maximum at index $i$ and $\tau_1$ would have less than $j$ elements while $\tau_2$ necessarily has at least $j$ elements.
	
	Now $\tau_1$ and $\tau_2$ must be equal after the first $j$ elements. Suppose that this common tail contains an element greater than any of the elements $\pi_i, \pi_{i+1}, \dots, \pi_{j-1}$. Then the insertion of a new maximum at index $j$ must have created an occurrence of $132$ that is still present in the common tail, which cannot be, so this is not the case. Suppose then this common tail contains $s$ right-to-left maxima.
	
	Now in $\tau_2$ the right-to-left maxima are these $s$ common right-to-left maxima and the newly inserted maximum element. In $\tau_1$ we must have the $s$ common right-to-left maxima, but also at least one right-to-left maxima among $\pi_i, \pi_{i+1}, \dots, \pi_{j-1}$ since the $s$ common right-to-left maxima are smaller than these elements. Furthermore $\tau_1$ also has a right-to-left maxima at index $i$ where the new maximum was inserted. Therefore $\tau_1$ and $\tau_2$ have different numbers of right-to-left maxima which completes our proof.
\end{proof}

\begin{theorem}
Denote the number of edges in the 1324-avoider graph going from an element of $A(n, r)$ to an element of $A(m, s)$ by $E(n, r, m, s)$. That function satisfies
\[E(n, r, m, s) = \begin{cases}
0 \text{ if } r \geq n, s \geq m, n < 1, m < 1, r < 0, s < 0 \text{ or } m > n + 1 \\
0 \text{ if } n < m \text{ and } s < r \\
T_{n - 1, r} \text{ if } n + 1 = m \\
E(n - 1, r, m, s) + E(n, r - 1, m, s) \text{ otherwise}
\end{cases}\]
\end{theorem}

\begin{proof}
	The first case says there is no permutation with more short elements than total elements, there is no way to have a negative number of short elements, there is no way to create the empty permutation from a non-empty one and there is no way gain more than one element by inserting a single element.
	
	Next the second case. If $n < m$ we must have inserted a new maximum element and not removed anything due to occurrences of $132$. If $s < r$ it means we have fewer short elements after insertion than before insertion. The newly inserted maximum cannot be short, and inserting a new maximum cannot make something go from being not a right-to-left maxima to being one. Therefore this case is impossible and the result is zero.
	
	Next we move onto the third case. In this case we have inserted a new maximum and removed no elements afterwards. To remove no elements we must insert our maximum directly right of a right-to-left maxima. If we insert the new maximum to the right of what was previously a right-to-left maxima, it is a short value afterwards. So to land in $A(m, s)$ we know after which right-to-left maxima we have to insert, since we must create the correct number of short values. Thus each element of $A(n, r)$ has exactly one position where inserting lands us in $A(m, s)$, giving the base case.
	
	Finally we just have one case left. Due to Lemma \ref{differentanr} each $\pi \in A(n, r)$ has at most one insertion landing in $A(m, s)$. Thus the value $E(n, r, m, s)$ is simply the number of $\pi \in A(n, r)$ that have such an insertion. Furthermore as $m \leq n$ we must remove at least one element from $\pi$ after this insertion. We split into cases depending on whether $\pi$ ends with a $1$.
	
	Any $\pi'  \in A(n - 1, r)$ can be made into a permutation in $A(n, r)$ by adding a new minimum element at the end and increasing every other element by one, as this introduces a new right-to-left maxima, leaving $r$ unchanged. Furthermore since the insertion into $\pi$ must remove at least element, the corresponding insertion in $\pi'$ gives the same result. Thus if our $\pi \in A(n, r)$ ends with a $1$ we can map it to a unique element in $A(n - 1, r)$. 
	
	Now we move onto the case where $\pi$ does not end with $1$. Suppose the last element of $\pi$ is $x > 1$. Then $x - 1$ is in $\pi$, suppose that the smallest element preceding or equal to $x - 1$ is $y$. Then we create $\pi'$ by increasing every element equal to at least $y$ and less than $x$ by one and replacing the last element by $y$, see Figure~\ref{xshift}.
	
		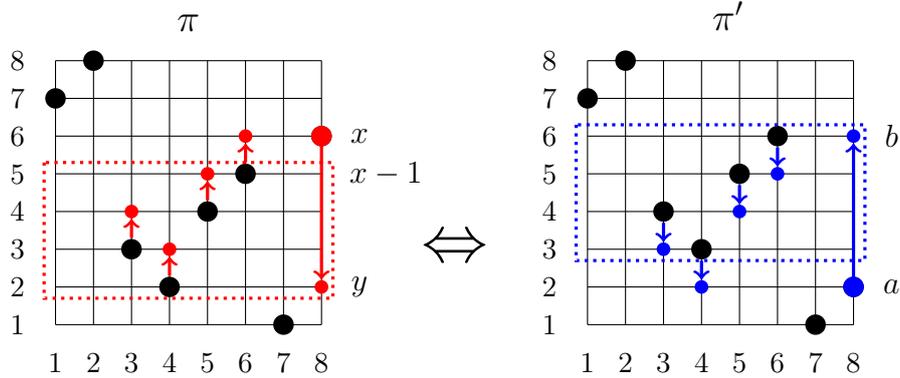
\begin{figure}[h]
		\centering
		\begin{tikzpicture}
			\tikzset{dot/.style={fill=black,circle}}
			
			\foreach\l[count=\y] in {1,2,...,8}
			{
				\draw[thin] (0.5,\y/2) -- (4,\y/2);
				\node at (0,\y/2){\l};
			}
			
			\foreach \x in {1,2,...,8}
			{
				\draw[thin] (\x/2,0.5) -- (\x/2,4);
				\node at (\x/2,0){\x};
			}
			
			\foreach \x/\y in {1/7,2/8,3/3,4/2,5/4,6/5,7/1,8/6}{
				\node[dot, scale=0.75] at (\x/2, \y/2) {};
			}
			
			\node[dot, red, scale=0.75] at (8/2, 6/2) {};
			\node[dot, red, scale=0.5] at (8/2, 2/2) {};
			
			\node (pi) at (2.25, 4.5) {\Large $\pi$};
			\node (x) at (4.5, 6/2) {\large $x$};
			\node (x2) at (4.85, 5/2) {\large $x - 1$};
			\node (y) at (4.5, 2/2) {\large $y$};
			
			%\draw[->, very thick, blue] (x2) -- (6/2+0.2,4/2);
			%\draw[->, very thick, blue] (6/2-0.2,4/2) -- (5/2,4/2) -- (5/2,2/2+0.2);
			%\draw[->, very thick, blue] (5/2+0.2,2/2) -- (y);
			
			\draw[->, very thick, red] (8/2,6/2-0.1) -- (8/2,2/2+0.1);
			\draw[very thick, red, dotted] (1/2-0.15,5/2+0.15) -- (1/2-0.15, 2/2-0.15) -- (8/2+0.15, 2/2-0.15) -- (8/2+0.15, 5/2+0.15) -- cycle;
			% \draw[->, very thick, red] (4.5/2,4/2+0.15) -- (4.5/2, 5/2+0.15);
			
			\node[dot, red, scale=0.5] at (3/2, 4/2) {};
			\node[dot, red, scale=0.5] at (4/2, 3/2) {};
			\node[dot, red, scale=0.5] at (5/2, 5/2) {};
			\node[dot, red, scale=0.5] at (6/2, 6/2) {};
			
			\draw[->, very thick, red] (3/2,3/2+0.15) -- (3/2,4/2-0.1);
			\draw[->, very thick, red] (4/2,2/2+0.15) -- (4/2,3/2-0.1);
			\draw[->, very thick, red] (5/2,4/2+0.15) -- (5/2,5/2-0.1);
			\draw[->, very thick, red] (6/2,5/2+0.15) -- (6/2,6/2-0.1);
			
			\foreach\l[count=\y] in {1,2,...,8}
			{
				\draw[thin] (7.5,\y/2) -- (11,\y/2);
				\node at (7,\y/2){\l};
			}
			
			\foreach \x in {1,2,...,8}
			{
				\draw[thin] (\x/2+7,0.5) -- (\x/2+7,4);
				\node at (\x/2+7,0){\x};
			}
			
			\foreach \x/\y in {1/7,2/8,3/4,4/3,5/5,6/6,7/1,8/2} {
				\node[dot, scale=0.75] at (\x/2+7, \y/2) {};
			}
			
			\node (pi2) at (2.35+7, 4.6) {\Large $\pi'$};
			
			\node (b) at (4.5+7, 6/2) {\large $b$};
			\node (a) at (4.5+7, 2/2) {\large $a$};
			
			\node[dot, blue, scale=0.75] at (8/2+7, 2/2) {};
			\node[dot, blue, scale=0.5] at (8/2+7, 6/2) {};
			\draw[<-, very thick, blue] (8/2+7,6/2-0.1) -- (8/2+7,2/2);
			
			\draw[->, very thick, blue] (3/2+7,4/2-0.15) -- (3/2+7,3/2+0.1);
			\draw[->, very thick, blue] (4/2+7,3/2-0.15) -- (4/2+7,2/2+0.1);
			\draw[->, very thick, blue] (5/2+7,5/2-0.15) -- (5/2+7,4/2+0.1);
			\draw[->, very thick, blue] (6/2+7,6/2-0.15) -- (6/2+7,5/2+0.1);
			
			\node[dot, blue, scale=0.5] at (3/2+7, 3/2) {};
			\node[dot, blue, scale=0.5] at (4/2+7, 2/2) {};
			\node[dot, blue, scale=0.5] at (5/2+7, 4/2) {};
			\node[dot, blue, scale=0.5] at (6/2+7, 5/2) {};
			
			\draw[very thick, blue, dotted] (1/2-0.15+7,6/2+0.15) -- (1/2-0.15+7, 3/2-0.15) -- (8/2+0.15+7, 3/2-0.15) -- (8/2+0.15+7, 6/2+0.15) -- cycle;
			
			\node (pi2) at (5.75, 1.5) {\Huge $\Leftrightarrow$};
			
		\end{tikzpicture}
		\caption{Map from $\pi$ to $\pi'$ when the last element is $> 1$}
		\label{xshift}
	\end{figure}
	
	Clearly no value $z > x$ can exist between $x - 1$ and $x$ in $\pi$ as $x - 1, z, x$ would be an occurrence of $132$. So $x - 1$ becomes a right-to-left maxima in $\pi'$ and was not one before in $\pi$, while $x$ was and remains a right-to-left maxima. Furthermore there can be no element $z$ satisfying $z > x - 1$ between $x - 1$ and $x$ as otherwise $y, z, x$ would be an occurrence of $132$. So no other element between becomes a right-to-left maxima, so the number of short values decreases by exactly one. Thus $\pi' \in A(n, r - 1)$.
	
	Going the other way we can start with $\pi \in A(n, r - 1)$. If $\pi$ is the increasing permutation, any insertion results in no element being removed afterwards, so it cannot end in $A(m, s)$. Therefore we can discount the case where $\pi$ is increasing, in the other case we have some last element $a$ and previous right-to-left maxima $b$. Then we can get the inverse by decreasing every element with value between $a$ and $b$ by one and making the last element $b$, completing the proof.
\end{proof}

This could likely be reduced to a recurrence in less than four variables. But for our purposes we need all the values $E(n, r, m, s)$ anyway and this calculates each of them in constant time, so this will suffice.

If we repeat our earlier constructions to create the avoider graph for $1324$ and take the quotient with respect to $A(n, r)$, checking some examples by hand reveals that the walks do not correspond one-to-one like before. In an attempt to still make the construction work, we introduce one new feature. The edge from $A(n, r)$ to $A(m, s)$ will have weight $E(n, r, m, s) / |A(n, r)|$, where we omit zero weight edges. A walk will then no longer contribute $1$ to the total sum, but the product of the weights on the edges it travels along. This is done to replicate the earlier behaviour of forcing two coordinates belonging to the same $A(n, r)$ to have the same weight, just without the luxury of always having the same number of edges going out of a partition to each outneighbour. For this reason we choose the weight  $E(n, r, m, s) / |A(n, r)|$, as it is the ratio of vertices in $A(n, r)$ having an edge pointing to $A(m, s)$. In previous cases, this ratio was always one. To distinguish these weighted walks from the normal ones, we denote the sum of the length $k$ walks in the graph with cutoff $n$ by $\widetilde{W}_{n,k}$ in the weighted case. We conjecture, but unfortunately do not prove, that these weighted walks provide a lower bound for the unweighted ones. 

Figure \ref{wtable} shows $\widetilde{W}_{n,k}/W_{n,k}$ for some values of $n$ and $k$. We start at $n = 4$ and $k = 6$ as all values before that point are simply $1$. All values for $k \leq 100$ and $n \leq 15$ have been confirmed to be $\leq 1$. As all paths of length $\leq \ell$ steps are all the same in graphs with cutoff $\geq \ell$, this means the conjecture is true for $k \leq 15$ for all $n$.

\begin{figure}[t]
\begin{center}
\begin{tabular}{|l|ccccc|}
	\hline
	\diaghead{\theadfont $n$ \, \,}{\, \, $k$}{$n$ \, \,}
	& \thead{6} & \thead{7} & \thead{8} & \thead{9} & \thead{10} \\ \hline
	\thead{4} & 0.999906 & 0.999893 & 0.999819 & 0.999747 & 0.999677 \\
	\thead{5} & 0.999971 & 0.999833 & 0.999605 & 0.999333 & 0.999053 \\
	\thead{6} & 0.999974 & 0.999864 & 0.999572 & 0.999111 & 0.998545 \\
	\thead{7} & 0.999975 & 0.999873 & 0.999622 & 0.999126 & 0.998388 \\
	\thead{8} & 0.999975 & 0.999875 & 0.999638 & 0.999192 & 0.998450 \\
	\thead{9} & 0.999975 & 0.999875 & 0.999641 & 0.999212 & 0.998523 \\ \hline
\end{tabular}
\end{center}
\caption{Values of $\widetilde{W}_{n,k}/W_{n,k}$ for various $n, k$.}
\label{wtable}
\end{figure}

Formally speaking, the conjecture to be proven is the following.

\begin{conjecture}
	\label{conj}
The growth rate of walks in the grouped 1324-avoiding graph provide a lower bound for the Stanley-Wilf limit of $1324$-avoiding permutations. That is to say $\widetilde{W}_{n, k}(1324) \leq W_{n,k}(1324)$ for all $n, k \geq 0$.
\end{conjecture}

Now if this were true, we get the following for free based on the computations involved. Though the above conjecture is in fact quite strong, as the corollary only needs the case $n = k$. The case $n = k$ has been verified to be true for $n \leq 50$ using the known values of $|\Av_n(1324)|$, provided by Conway, Guttmann and Zinn-Justin \cite{estimate}. A plot of the true value compared to the esimate can be seen in Figure~\ref{estimate}. The eagle-eyed viewer might see that the two values diverge very slightly as they get to index $50$. On a non-logarithmic plot the difference somewhat clearer as the estimate is about 30\% smaller than the true value at $n = 50$.

\begin{figure}[htb]
	\centering
	\includegraphics[width=\textwidth]{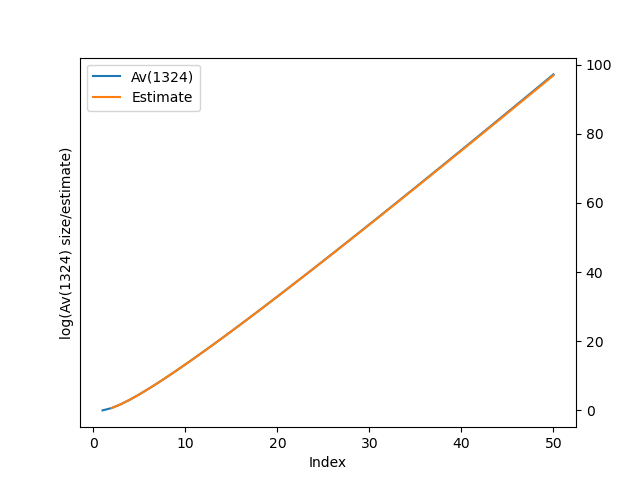}
	\caption{Logarithm of estimate of $|\Av_n(1324)|$ compared to true value}
	\label{estimate}
\end{figure}

\begin{corollary}
If Conjecture \ref{conj} is true, the growth rate of $1324$-avoiding permutations is at least $10.418$.
\end{corollary}

The computed value for $n = 220$ is $\geq 10.418$, though the true bound for $n = 220$ is likely a bit better. For the largest values of $n$, it was difficult to have the power method converge fully due to computation constraints. But thanks to the Collatz-Wielandt formula we can still check and be sure that the bound $10.418$ is definitely correct. If we plot the three examples we have looked at, it seems like the converge speed as a function of the cutoff $N$ near the limit seems to be approximately $N^{-3/2}$, see Figure~\ref{convergence}. 

\begin{figure}[htb]
	\centering
	\includegraphics[width=\textwidth]{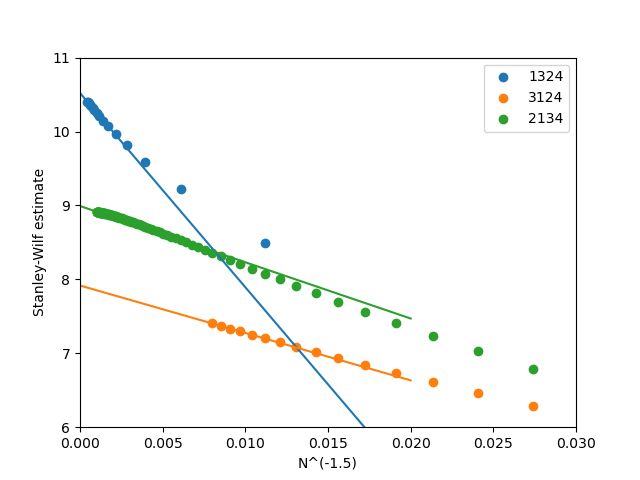}
	\caption{Plot of convergence of estimates to their limit}
	\label{convergence}
\end{figure}

It seems that this method is a genuine lower bound and does not tend to the conjectured growth rate of $11.6$. But this gap could partially be a result of our limited floating-point precision method failing to find good solutions for larger cases. A guess based on how quickly the value seems to be converging for $n \leq 220$ is that the limit is about $10.65$. So with more computation the value $10.418$ could be improved, but maybe not by much. With much more computation time this could also be repeated using higher precision floating points, which could reveal how much of this gap is a result of computational constraints.

\section{Open questions}

\begin{itemize}
	\item What properties do these sequences of increasing graphs need to satisfy for the growth rates of their walks to converge to the growth rate of walks in the limit graph? Since the graphs are directed there are some obvious counterexamples to the limit always being correct if we place no constraints. For example an infinite binary tree with edges oriented from parent to child would have growth rate $0$ for any finite section, but growth rate $2$ in the limit. Does it suffice that each graph in the sequence is strongly connected? This question could also be equivalently posed in terms of automata to fit with earlier work, like the paper by Albert, Elder, Rechnitzer, Wextcott and Zabrocki \cite{lockseq}.
	
	\item For a stationary distribution of the random walks on one of these (non-truncated) graphs, does there always exist an $\alpha \in \mathbb{R}$ and integer valued function $f$ on permutations so the coordinate for $\pi$ is $\alpha f(\pi)/|\pi|!$ for all $\pi$?
	
	\item What is the theoretical connection between the stationary distributions of the random walks and the Perron-Frobenius eigenvector of the graph? They are both eigenvectors of a matrix related to the graph, but not of the same matrix. Yet they seem to share some structural properties, for example being constant on the same partition in some of the simpler examples considered.
	
	\item Is there some way to compute or approximate the Perron-Frobenius eigenvector in the infinite graph? Simply computing a coordinate for larger and larger cutoffs $N$ does not even clearly reveal simple things like whether two coordinates are equal in the limit due to ``noise'' emanating from where the graph was truncated.
	
	\item Is there some other partition that gives a better bound in the case of $1324$-avoiding permutations, while still being computationally feasible?
\end{itemize}

\section{Acknowledgements}

Figuring out the formula for the weights for the random walks was done using the FindStat \cite{FindStat} tool.

\section{Appendix}

The calculated values of $\widetilde{W}_{n,n}$ for $n = 1, \dots, 50$ follow in Figure~\ref{appenddata}. These can be compared against the known values of $|\Av_n(1324)|$ from Conway, Guttmann and Zinn-Justin \cite{estimate}.

\begin{figure}[ht]
\centering
\begin{tabular}{|r|l|}
	\hline
	$n$ & $\widetilde{W}_{n,n}$ \\ \hline
	$1$ & \texttt{0} \\
	$2$ & \texttt{2} \\
	$3$ & \texttt{6} \\
	$4$ & \texttt{23} \\
	$5$ & \texttt{103} \\
	$6$ & \texttt{513} \\
	$7$ & \texttt{2762} \\
	$8$ & \texttt{15792.6} \\
	$9$ & \texttt{94764.14143} \\
	$10$ & \texttt{591737.5476} \\
	$11$ & \texttt{3821110.811} \\
	$12$ & \texttt{25394500.09} \\
	$13$ & \texttt{173036190} \\
	$14$ & \texttt{1205205579} \\
	$15$ & \texttt{8559183937} \\
	$16$ & \texttt{6.185211294e+10} \\
	$17$ & \texttt{4.540198358e+11} \\
	$18$ & \texttt{3.380289108e+12} \\
	$19$ & \texttt{2.549431847e+13} \\
	$20$ & \texttt{1.945671483e+14} \\
	$21$ & \texttt{1.501133141e+15} \\
	$22$ & \texttt{1.169850253e+16} \\
	$23$ & \texttt{9.202016022e+16} \\
	$24$ & \texttt{7.301188881e+17} \\
	$25$ & \texttt{5.839947044e+18} \\ \hline
\end{tabular}
\begin{tabular}{|r|l|}
	\hline
	$n$ & $\widetilde{W}_{n,n}$ \\ \hline
	$26$ & \texttt{4.706543085e+19} \\
	$27$ & \texttt{3.820046813e+20} \\
	$28$ & \texttt{3.121218687e+21} \\
	$29$ & \texttt{2.566262389e+22} \\
	$30$ & \texttt{2.122495423e+23} \\
	$31$ & \texttt{1.765314252e+24} \\
	$32$ & \texttt{1.476043875e+25} \\
	$33$ & \texttt{1.240398630e+26} \\
	$34$ & \texttt{1.047369958e+27} \\
	$35$ & \texttt{8.884154482e+27} \\
	$36$ & \texttt{7.568625284e+28} \\
	$37$ & \texttt{6.474664471e+29} \\
	$38$ & \texttt{5.560801289e+30} \\
	$39$ & \texttt{4.794057984e+31} \\
	$40$ & \texttt{4.148068322e+32} \\
	$41$ & \texttt{3.601642999e+33} \\
	$42$ & \texttt{3.137668938e+34} \\
	$43$ & \texttt{2.742259729e+35} \\
	$44$ & \texttt{2.404096157e+36} \\
	$45$ & \texttt{2.113911500e+37} \\
	$46$ & \texttt{1.864087697e+38} \\
	$47$ & \texttt{1.648336933e+39} \\
	$48$ & \texttt{1.461449357e+40} \\
	$49$ & \texttt{1.299092264e+41} \\
	$50$ & \texttt{1.157649550e+42} \\ \hline
\end{tabular}
\caption{$\widetilde{W}_{n,n}$ for $n = 1, \dots, 50$}
\label{appenddata}
\end{figure}

\bibliographystyle{acm}

\bibliography{machinepaper}

\end{document}